\documentclass[a4paper]{amsart}

\pdfoutput=1

\usepackage[utf8]{inputenc}

\usepackage{amsthm, amssymb, amsmath, amsfonts, mathrsfs}
\usepackage[colorlinks=true, pdfstartview=FitV, linkcolor=blue, citecolor=blue, urlcolor=blue,pagebackref=false]{hyperref}

\usepackage{microtype}

\usepackage{color}

\newtheorem{thm}{Theorem}[section]
\newtheorem{prop}[thm]{Proposition}
\newtheorem{lemma}[thm]{Lemma}
\newtheorem{cor}[thm]{Corollary}
\theoremstyle{remark}
\newtheorem{rem}[thm]{Remark}

\renewcommand{\le}{\leqslant}
\renewcommand{\ge}{\geqslant}
\renewcommand{\leq}{\leqslant}
\renewcommand{\geq}{\geqslant}
\newcommand{\Ll}{\left}
\newcommand{\Rr}{\right}
\renewcommand{\subset}{\subseteq}
\newcommand{\mcl}{\mathcal}
\newcommand{\msf}{\mathsf}
\newcommand{\mfk}{\mathfrak}

\newcommand{\E}{\mathbb{E}}

\newcommand{\EE}{\mathbf{E}}

\newcommand{\N}{\mathbb{N}}

\newcommand{\1}{\mathbf{1}}
\newcommand{\R}{\mathbb{R}}
\newcommand {\IR}{\mathbb{R}}

\newcommand{\Z}{\mathbb{Z}}
\newcommand {\IZ}{\mathbb{Z}}
\renewcommand{\P}{\mathbb{P}}
\newcommand{\PP}{\mathbf{P}}
\newcommand{\ov}{\overline}
\newcommand{\td}{\tilde}
\newcommand{\eps}{\varepsilon}
\renewcommand{\epsilon}{\varepsilon}
\def\d{{\mathrm{d}}}

\renewcommand{\phi}{\varphi}

\renewcommand{\b}{_\lambda}

\renewcommand{\emptyset}{\varnothing}
\newcommand{\un}{\underline}

\numberwithin{equation}{section}

\title[Lyapunov exponents of random walks in small random potential]{Lyapunov exponents of random walks in small random potential: the upper bound}
\author{Thomas Mountford, Jean-Christophe Mourrat}

\address[Thomas Mountford]{Ecole polytechnique fédérale de Lausanne, Institut de mathé\-ma\-tiques, station 8, 1015 Lausanne, Switzerland}

\address[Jean-Christophe Mourrat]{Ecole polytechnique f\'ed\'erale de Lausanne, Institut de math\'ematiques, station 8, 1015 Lausanne, Switzerland \& ENS Lyon, CNRS, 46 allée d'Italie, 69007 Lyon, France}

\begin{document}
\begin{abstract}
We consider the simple random walk on $\Z^d$ evolving in a random i.i.d.\ potential taking values in $[0,+\infty)$. The potential is not assumed integrable, and can be rescaled by a multiplicative factor $\lambda > 0$. Completing the work started in a companion paper, we give the asymptotic behaviour of the Lyapunov exponents for $d \ge 3$, both annealed and quenched, as the scale parameter $\lambda $ tends to zero.

\bigskip


\medskip

\end{abstract}

\maketitle

%
%
%
%
%
%
%
\section{Introduction}

Let $(X_n)_{n \in \N}$ be the simple random walk in $\Z^d$, whose law starting from $x$ we write $\PP_x$ (with expectation $\EE_x$). Independently of the random walk, we give ourselves a random potential $V = (V(x))_{x \in \IZ^{d}}$, which is a family of i.i.d.\ random variables taking values in $[0,+\infty)$.  We write $\P = \mu^{\otimes \IZ^d}$ for the law of this family, with associated expectation $\E$.
Let $\ell \in \IR^d$ be a vector of unit Euclidean norm, and
$$
\msf{T}_n(\ell) = \inf \Ll\{ k : X_k \cdot \ell \ge n \Rr\}
$$ 
be the first time at which the random walk crosses the hyperplane orthogonal to~$\ell$ lying at distance $n$ from the origin. For every $\lambda > 0$, we define the \emph{quenched} and \emph{annealed} point-to-hyperplane Lyapunov exponents associated with the direction $\ell$ and the potential $\lambda V$ by, respectively,
\begin{equation}
\label{defalpha}
{\alpha}_\lambda(\ell) = \lim_{n \to + \infty} - \frac{1}{n} \ \log \EE_0 \Ll[ \exp\Ll( -\sum_{k = 0}^{\msf{T}_n(\ell)-1} \lambda V(X_k) \Rr) 
\Rr],
\end{equation}
\begin{equation}
\label{defalphabar}
\ov{\alpha}_\lambda(\ell) = \lim_{n \to + \infty} - \frac{1}{n} \ \log \E\EE_0\Ll[ \exp\Ll( -\sum_{k = 0}^{\msf{T}_n(\ell)-1} \lambda V(X_k) \Rr) 
\Rr].
\end{equation}
The first limit holds almost surely, and is deterministic (see \cite{shape} for the existence of the first limit, and \cite{flu1} for the second one).

Our goal is to complete the proof of the following result.
\begin{thm}
\label{t:main}
If $d \ge 3$, then as $\lambda$ tends to $0$,
\begin{equation}
\label{e:main}
\alpha_\lambda(\ell) \sim \ov{\alpha}_\lambda(\ell) \sim \Ll(2d \, \int q_d \ \frac{ 1-e^{- \lambda z} }{1-(1- q_d) e^{- \lambda z}} \ \d \mu(z) \Rr)^{1/2},
\end{equation}
where $q_d$ is the probability that the simple random walk never returns to its starting point, and $a_\lambda \sim b_\lambda$ stands for $a_\lambda/b_\lambda \to 1$.
\end{thm}
We refer to \cite{shape1} for a detailed review of previous results, for motivations and for a heuristic explanation of this formula. 
We define $f : \R_+ \to \R$ by
\begin{equation}
\label{e:deff}
f(z) = q_d \ \frac{ 1-e^{-  z} }{1-(1- q_d) e^{-  z}},
\end{equation}
so that the integral appearing in the right-hand side of \eqref{e:main} can be rewritten as
\begin{equation}
\label{defI}
\mfk{I}_\lambda =  \int f(\lambda z) \ \d\mu (z).
\end{equation}
It was shown in \cite{kmz} that if $\E[V] < + \infty$ (we use $\E[V]$ as shorthand for $\E[V(0)]$), then as $\lambda$ tends to $0$,
$$\alpha_\lambda(\ell) \sim\ov{\alpha}_\lambda(\ell) \sim \Ll(2d \, \lambda \, \E[V]\Rr)^{1/2}.
$$
Since the function $f$ is concave and $f(z) \sim z$ as $z$ tends to $0$, this is consistent with Theorem~\ref{t:main}. For $\E[V] = +\infty$ and $d \ge 3$, it was shown in \cite{shape1} that
\begin{equation}
\label{l:shape1}
\liminf_{\lambda \to 0} \frac{\ov{\alpha}_\lambda(\ell) }{\sqrt{2d \, \mfk{I}_\lambda} }  \ge 1.
\end{equation}
A standard convexity argument ensures that $\ov{\alpha}_\lambda(\ell) \le \alpha_\lambda(\ell)$. As a consequence, in order to prove Theorem~\ref{t:main}, it suffices to show the following.
\begin{thm}
\label{t:main2}
If $d \ge 3$ and $\E[V] = +\infty$, then 
$$
\limsup_{\lambda \to 0} \frac{{\alpha}_\lambda(\ell) }{\sqrt{2d \, \mfk{I}_\lambda} }  \le 1.
$$
\end{thm}
Our task here is thus to show Theorem~\ref{t:main2}. Its proof is simpler than that of the converse bound in \eqref{l:shape1}.  Indeed, instead of having to consider all possible combinations of paths and environments, we must simply, given a typical environment, construct a scenario whose probability is appropriately bounded from below and for which the walk travels to the distant hyperplane.  As a first step, we use the following observation, due to \cite{bk,zer}.


\begin{lemma} 
\label{lemconst}
For every $\eps > 0$, let $\td{V}_\eps = (\td{V}_\eps(x))_{x \in \Z^d}$ be the potential defined by
$$
\tilde{V}_{\epsilon}(x) =
\left|
\begin{array}{ll}
V(x) & \text{if}\hspace{0.3cm} V(x) \geq \frac{\epsilon}{\lambda},\\
\E \Ll[ V(0) \ \vert \ V(0) < \frac{\epsilon}{\lambda}\Rr] &  \text{if}\hspace{0.3cm} V(x) < \frac{\epsilon}{\lambda},
\end{array}
\right.
$$
and $\alpha_{\eps,\lambda}(\ell)$ be the quenched Lyapunov exponent associated with the potential $\td{V}_\eps$. We have $\alpha_{\eps,\lambda}(\ell) \ge \alpha_{\lambda}(\ell)$.
\end{lemma}
This follows as in the proof of the last statement of \cite[Proposition~4]{zer}, using the fact that the convergence in \cite[(6.12)]{shape} holds in $L^1$. We define 
\begin{equation}
\label{defIeps}
\mfk{I}_{\eps,\lambda} = \int f(\lambda z) \ \d \mu_\eps(z),
\end{equation}
where $\mu_\eps$ is the law of $\td{V}_\eps(0)$. Elementary 
bounds yield that $\mfk{I}_{\eps,\lambda} \leq (1+ C \epsilon) \mfk{I}_\lambda$ for some universal constant $C$, and thus Theorem~\ref{t:main2} will be a consequence of
\begin{prop}
\label{p:main}
There exists $ K < \infty $ (independent of $\eps$) such that for $\alpha_{\eps,\lambda}(\ell)$ as above,
$$\limsup _{\lambda \to 0 }\frac{\alpha_{\eps,\lambda}(\ell)}{ \sqrt{2d \, \mfk{I}_{\eps,\lambda}}} \leq 1 + K \epsilon.$$
\end{prop}

The advantage of this reduction is that it permits us to deal with a simpler environment than the original one.  For $\lambda $ small the great majority of points $x \in \IZ^d$ will have
$\tilde V_ \epsilon (x) $ equal to the constant value $ \E [ V(0) \ \vert \ V(0) < \frac{\epsilon}{\lambda}]$.  We call the remaining points (that is, the sites $x$ such that $V(x) \ge \eps/\lambda$) the \emph{important points}. Those sites will (when suitably renormalized) ressemble a Poisson cloud, and can be analysed by coarse-graining techniques.
\vspace{0.3cm}

In this paper, the main work is done in Section 3.  This is preceded by Section 2 which gives some simple technical results for random walks, chiefly based on the invariance principle.
Section 3 exploits the law of large numbers for the environment as $\lambda $ becomes small.  It culiminates in Proposition \ref{stage} which states that within a ``good" environmnent, the random walk can move forward in the $\ell$-direction while encurring appropriate costs. This ``building block" is transformed into a result about the Lyapunov exponent by a block argument in Section 4. Section 5 concerns a special case that had to be left appart in the previous arguments.

For notational reasons, we will treat explicitly the case $\ell \ = \ (1,0, \cdots 0)$, but this entails no loss of generality, since the main tool is the invariance principle, and there are no subtle lattice effects to take account of.




%

%
%
%
%
%
%
%
\section{Technical estimates}

In this section, we wish to estimate the probability for our random walk to advance by $n$ in the $\ell$ direction, with appropriate speed.
%
%
%
%
%
%
We first analyze the limiting object for rescaled random walks, that is, Brownian motion with a drift. We are interested in Brownian motion with covariance matrix $\frac{1}{d} \ \mathrm{Id}$, since it is the scaling limit of our discrete-time random walk.\\

Let $B = (B_t)_{t \ge 0}$ denote the canonical process on the space $C(\R_+,\R^d)$ of continuous functions from $\R_+$ to $\R^d$, and let $\mcl{F}_t = \sigma(B_s, s \le t)$. For $x \in \R^d$ and $h \in \R$, we denote by $Q^h_x$ the law of the Brownian motion with covariance matrix $\frac{1}{d} \ \mathrm{Id}$, drift $h \ell$ and starting at $x$, where $\ell = (1,0,\ldots,0) \in \R^d$. We simply write $Q_x$ for $Q^0_x$. As is well-known, the measure $Q^h_x$ has Radon-Nikodym derivative
\begin{equation}
\label{Radon-Nikodym}
e^{dh \ell \cdot (B_t - B_0) - dh^2 t/2}
\end{equation}
with respect to the measure $Q_x$ on $\mcl{F}_t$ (see \cite{RY}).

For every $M \ge 1$, consider the event $A^h_{M}$ defined by
\begin{enumerate}
\item[(i)] 
the path hits $\{M \} \times \left( \frac{\sqrt{M}}{2}, \frac{3\sqrt{M}}{2} \right) \times  \left( \frac{-\sqrt{M}}{2}, \frac{\sqrt{M}}{2} \right)^{d-2}$ in time $M/h$ or less,
\item[(ii)]
before this time, the path $B$ does not leave $(B_0-2, \infty) \times
\left( - \sqrt{M}, 2\sqrt{M} \right) \times { \left( -\sqrt{M}, \sqrt { M} \right)}^{d-2} $. 
\end{enumerate} 

Since under $Q^h_x$, the first component essentially moves linearly with speed $h$ while the other components vary diffusively, one has 
\begin{lemma}
For every $h \neq 0$, there exists a constant $c = c(d,h) > 0$ such that for every $M \ge 1$ and $x \in \left( -\sqrt{M} /{2} , \sqrt{M}/{2} \right)^{d}$,
$$Q_x^h\Ll[A^h_{M}\Rr] \geq c.$$
\end{lemma}
This and the Radon-Nikodym derivative given in \eqref{Radon-Nikodym} yield
\begin{cor}
\label{c:bm}
For every $h \neq 0$, there exists a constant $c_1 = c_1(d,h) > 0$ such that for every $M \ge 1$ and $x \in (-1/2,1/2) \times \left( -\sqrt{M} /{2} , \sqrt{M}/{2} \right)^{d-1}$,
$$
Q_x\Ll[A^h_{M}\Rr] \geq c_1 e^{-M \frac{dh}{2}}.
$$
\end{cor}

This result can now be applied to the original object of interest, the random walk. On the space of random walk trajectories $X$, we define the event $A^{h,L}_M$ by
\begin{enumerate}
\item[(i)] 
the random walk hits $\{M L \}  \times { \left( \frac{\sqrt{M}L}{2}, \frac{3\sqrt{M}L}{2} \right) } \times { \left( \frac{-\sqrt{M}L}{2},  
\frac{\sqrt{M}L}{2} \right)}^{d-2}$ 
in time $M L^{2}/{h}$ or less,
\item[(ii)]
before this time, the walk $X$ does not leave $(X_0-2L, \infty ) \times \left( - \sqrt{M}L, 2\sqrt{M}L \right) \times { \left( -\sqrt{M}L, \sqrt { M} L \right)}^{d-2}$. 
\end{enumerate}

From the invariance principle, we thus get the following.
\begin{cor}
\label{c:rw}
Let $h \neq 0$ and $c_1 > 0$ be given by Corollary~\ref{c:bm}. For every $M \ge 1$, every $L$ sufficiently large, and every $x \in (-L/2,L/2) \times \left( -\sqrt{M}L /{2} , \sqrt{M}L/{2} \right)^{d-1}$,
$$
\PP_x\Ll[A^{h,L}_{M}\Rr] \geq \frac{2c_1}{3} e^{-M \frac{dh}{2}}.
$$
\end{cor}

We conclude this section with a technical variation of this result that will be more adapted to our needs.
Given $0 < \epsilon_{0} \le \eps$ and $L$, we define the stopping times ${(\sigma_{i})}_{i \geq 0}$ recursively by $\sigma_{0}=0$ and, for $i \ge 1$,
\begin{equation}
\label{e:defsigmai}
\sigma_{i} = \inf \{n > \sigma_{i-1} : \vert X_{n} -X_{\sigma_{i-1} } \vert \geq \epsilon_{0} L\}.
\end{equation}
These stopping times are introduced as a substitute for  fixed times.
The point is that, given $M$, we can choose the parameter $\epsilon_0 $ to be sufficiently small that the $\sigma_i$ exhibit good behaviour even on the ``extreme" event $A^{h,L}_M$.
Let $\td{A}^{h,L}_{M}$ be the event that 
\begin{enumerate}
\item[(i)] 
the random walk hits $\{M L\} \times \left( \frac{\sqrt{M}L}{2}, \frac{3\sqrt{M}L}{2} \right) \times { \left( \frac{-\sqrt{M}}{2}, \frac{\sqrt{M}L}{2} \right)}^{d-2}$ before time $\sigma _{\frac{M}{h}\frac{ (1+\epsilon)}{\epsilon_{0}^{2}} }$,
\item[(ii)]
before this time, the walk $X$ does not leave $\left(X_0-2 L, \infty\right) \times \left(- \sqrt{M} L, 2 \sqrt{M} L\right) \times \left(- \sqrt{M} L,  \sqrt{M} L\right)^{d-2}$. 
\end{enumerate} 
Let us denote by $F$ the (random) smallest index $k$ such that $\sigma_k $ is at least as large as the hitting time of the hyperplane $\{x: x_1 = ML\}$.  So on the event $\td{A}^{h,L}_{M}$, we have $F \ \leq \frac{M}{h} \frac{(1+ \epsilon )}{\epsilon_0 ^ 2}$.
\begin{lemma}
\label{six}
Let $h \neq 0$ and $c_1 > 0$ be given by Corollary~\ref{c:bm}. For every $M \ge 1$, every $L$ large enough, every $\eps_0 > 0$ small enough and every $x \in (-L/2, L/2) \times \left( -\sqrt{M}L /{2} , \sqrt{M}L/{2} \right)^{d-1}$,
$$
\PP_x\Ll[\td{A}^{h,L}_{M}\Rr] \ge \frac{c_1}{2} e^{-M \frac{dh}{2}}.
$$
\end{lemma}
\begin{proof}
This follows from the fact that
\begin{eqnarray*}
\PP_x\Ll[A^{h,L}_{M} \setminus \td{A}^{h,L}_{M}\Rr] & \leq & \PP_x\Ll[\sigma_{ \frac{M}{h} \frac{(1+\epsilon)}{\epsilon_{0}^{2}} } < \frac{M L^{2}}{h}\Rr] \\
& \leq & e^{- \frac{hM}{2} \frac{C \epsilon }{\epsilon_{0}^{2}} },
\end{eqnarray*}
which is arbitrarily small compared with $e^{-M \frac{dh}{2}}$ provided that we choose $\eps_0$ sufficiently small.
\end{proof}

\begin{rem}
From now on we will take $h = \frac{1}{\sqrt d}$, and the superscript $h$ will be dropped in any notation for events or variables.
\end{rem}
\begin{rem}
From now on we suppose $\epsilon_{0}$ to be small enough and $L$ large enough to ensure that the conclusion of Lemma \ref{six} is true.  We also introduce
$ \delta $ and $\delta _1$ so that $0 < \delta_1 \ll \delta \ll \epsilon_0$.
They will need to be small enough to satisfy a finite number of conditions given below, but are otherwise kept fixed.
\end{rem}


%
%
%
%
%
%
%
\section{Coarse Graining and the Environment}
  
In this section we begin to consider the environment.  Our first task is to show that on the scale $L\b$, chosen as below, the environment is highly regular for  points of high $V(.)$ value and, given only ``reasonable" law of large numbers behaviour, these points will be such that they are struck in a ``Poisson" manner.


We now choose $L = L\b $ as a function of $\lambda$ as
\begin{equation}
\label{defL}
L\b^{-1} = \sqrt{2 \mfk{I}_{\eps,\lambda}},
\end{equation}
where we recall that $\mfk{I}_{\eps,\lambda}$ was defined in \eqref{defIeps}. 
We aim to show that the essential features of the problem become visible at this scale, and that the useful random walk paths will behave at this scale as random with a bias.  At lower scales they will just be unbiased random walks, at higher they become deterministic motion.

We first suppose that
\begin{equation}
\label{assumption1}
L\b^{-1} = \sqrt{2 \mfk{I}_{\eps,\lambda}}\leq \epsilon^{-2} \sqrt{\P\Ll[V \geq \frac{\epsilon}{\lambda}\Rr]}.
\end{equation}

We will later sketch the (easier) second case where $\{x: x > \frac{\epsilon}{\lambda}\}$ makes little contribution to $\mfk{I}_\lambda$. An immediate consequence of this assumption is that
\begin{equation}
\label{c:assumption1}
\int_{\frac{\epsilon}{\lambda}}^{+\infty} f(\lambda z) \ \d \mu(z) \ge f(\eps) \P\Ll[V \ge \frac{\epsilon}{\lambda}\Rr] \ge \frac{\eps^5}{L_\lambda^2}.
\end{equation}

We divide up the values in $\left[ \frac{\epsilon}{\lambda}, \frac{1}{\epsilon \lambda}\right)$ into intervals of length $\frac{\epsilon^{2}}{\lambda}$ (except for the last) $I_{0} = [a_0,b_0), I_{1} = [a_1,b_1), \dots I_{R} = [a_R,b_R)$, and we let $I_{R+1} = \left[ \frac{1}{\epsilon \lambda}, \infty\right),$ the interval of values best avoided. Note that the number $R$ depends only on $\eps$ (which will be chosen sufficiently small but otherwise kept fixed), and not on $\lambda$ (which we will let tend to $0$).

We divide up the intervals into two classes, as in \cite{shape1}. We say that the interval $I_j$ is \emph{relevant} if
\begin{equation}
\label{defrelevant}
\P[V(0) \in I_{j}] \ge \frac{\epsilon^{9}}{L\b^{2}}.
\end{equation}
We say that it is \emph{irrelevant} otherwise.

As the name indicates, the key is that points with values in irrelevant intervals are not relevant to scale $L\b$, while the number of relevant important points in a ``good cube'' (to be specified later) of side length $L\b^d$ should be of order $L\b^{d-2}$, and so there should be a reasonable chance that one of these points will be hit by the random walk (or by a lightly conditioned random walk) before exiting the cube.

For $\delta_1 > 0$ (to be chosen much smaller than $\eps$, and otherwise kept fixed) and every $\underline{n} = (n_1,\ldots,n_d) \in \Z^d$, we define the mesoscopic box
\begin{equation}
\label{defC}
C_{\underline{n}, \delta_1} =  \big[\delta_1 L\b \underline{ n}, \delta_1 L\b (\underline{n} +\underline{1})\big) = \prod_{i = 1}^d \  \big[\delta_1 L\b n_i, \delta_1 L\b (n_i + 1)\big).
\end{equation}

We say that an environment $(V(x))_{x \in \Z^d}$ is $(L\b, \delta_1)${\it-good on }$A \subseteq \IZ^{d}$ (or that $A$ \emph{is }$(L\b,\delta_1)$\emph{-good}) if for every mesoscopic box $C_{\underline{n}, \delta_1}$ (with $\underline{n} \in \Z^d$) intersecting $A$, the following two properties hold:
\begin{itemize}
\item
for every $j$ such that $I_j$ is relevant,
\begin{equation}
\label{goodrel}
\sum_{x \in C_{\un{n},\delta_1}} \1_{V(x) \in I_j} \le (1 + \epsilon) (\delta_1 L\b)^{d} \ \P[V(0) \in I_{j}] ;
\end{equation}
\item
for every $j$ such that $I_j$ is irrelevant,
\begin{equation}
\label{goodirr}
\sum_{x \in C_{\un{n},\delta_1}} \1_{V(x) \in I_j} \le 2{(\delta_1 L\b)}^{d} \  \frac{\epsilon^{9}}{L\b^{2}}.
\end{equation}
\end{itemize}
We start by showing that sets with size of order $L\b$ are good with high probability.
\begin{lemma}  
\label{goodenv}
For every fixed $M, \epsilon $ and $\delta_1$, the probability that $[-ML\b, ML\b]^{d}$ is $(L\b, \delta_1)$-good tends to one as $\lambda$ tends to $0$.
\end{lemma}
\begin{proof}
For $M$, $\epsilon$ and $\delta_1 $ fixed, the number $(R+2)$ of relevant and irrelevant intervals to be considered remains bounded as $\lambda$ tends to $0$. Similarly, the number of cubes of the form $C_{\underline{n},\delta_1}$ with $\underline{n} \in \Z^d$ that intersect $[-ML\b, ML\b]^{d}$ remains bounded as $\lambda$ tends to $0$.


We consider a relevant interval $I_j$ and any cube $C_{\underline{n},\delta_1}  $.  The random quantity 
$$
 \sum_{x \in C_{\underline{n},\delta_1}} I _{V(x) \in I_j} 
$$
is a binomial random variable with parameters $|C_{\underline{n},\delta_1}|$ and $P[V(0) \in I_j]$.
This is stochastically dominated by a Binomial random variable with parameters $(L\b \delta_1 + 1)^d $ and $P[V(0) \in I_j]$.
By elementary bounds on Binomial tails (see for instance \cite[(2.16)]{shape1}), we have for $\lambda$ small enough,
\begin{multline}
P\Ll[\sum_{x \in C_{\underline{n},\delta_1}} I _{V(x) \in I_j} \ge (1+\eps) \ (\delta_1 L\b)^d \ P[V(0) \in I_j] \Rr]
\\ 
\le P\Ll[\sum_{x \in C_{\underline{n},\delta_1}} I _{V(x) \in I_j} \ge (1+\eps/2) \ |C_{\underline{n},\delta_1}| \ P[V(0) \in I_j] \Rr] \leq \ e^{-c( \epsilon ) |C_{\underline{n},\delta_1}|}
\end{multline}
for some strictly positive $c( \epsilon )$. 
A similar reasoning applies to irrelevant intervals. 
%
%
\end{proof}


The next result shows that in a good environment, the walk will typically not visit important points close to the boundary of $B(X_0, L\b \epsilon_0)$.
\begin{lemma}
\label{l:annulus}
There exists $K$ (depending on $\eps$) such that for every $\epsilon_0$ and $\delta$ satisfying $ 0 < \delta < \epsilon_0 /2$ and every $\delta_1$ small enough, the following holds. If the environment in $B(0,L\b(\epsilon_{0} + 3 \delta))$ is $(L\b, \delta_{1})$-good, then the probability that a random walk beginning at the origin hits a site of value $\frac{\epsilon}{\lambda}$ or more in $B(0,L\b(\epsilon_0 + 2 \delta)) \setminus B(0, L\b(\epsilon_0 - \delta))$ is bounded by  $K \delta \epsilon_{0}$.
\end{lemma}
\begin{proof}
Recalling \eqref{e:deff} and \eqref{defIeps}, we note that
$$
\mfk{I}_{\eps,\lambda} = \int f(\lambda z) \ \d \mu_ \epsilon (z) \ge f(\eps) \P[V(0) \ge \eps/\lambda],
$$
and in particular, by the definition of $L\b$, see \eqref{defL},
$$
\P[V(0) \ge \eps/\lambda] \le \frac{L\b^{-2}}{2 f(\eps)}.
$$
If $\delta_1$ is sufficiently small and the environment is $(L\b, \delta_{1})$-good, then it follows that the number of important points in  $B(0,L\b(\epsilon_0 + 2 \delta)) \setminus B(0, L\b(\epsilon_0 - \delta))$ is less than $(1 + \epsilon + 2 \epsilon ^ 2)  L\b^{-2}/f(\eps)$ times the total number of points that are in cubes $C_{\underline{n},\delta_1}$ intersecting this region.  Clearly for $ \delta_1 $ small 
this number is bounded 
by a universal constant times
$$
\frac{L\b^{-2}}{f(\eps)} \epsilon_{0}^{d-1} \delta L\b^{d}.
$$
Recall that, as given by \cite[Theorem~1.5.4]{law}, there exists $K_1 > 0$ such that
\begin{equation}
\label{boundgreen}
\PP_0[X \text{ visits } y] \le K_1 \ |y|^{2-d}.	
\end{equation}
Hence, the probability that a given point in $B(0,L\b(\epsilon_0 + 2 \delta)) \setminus B(0, L\b(\epsilon_0 - \delta))$ is touched by the random walk is thus bounded by
$$
\frac{K_{1}}{{(L\b(\epsilon_{0} - \delta))^{d-2}}} \leq \frac{K_{1} 2^{d-2}}{{(L\b\epsilon_{0})^{d-2}}}.
$$
The probability described in the lemma is thus bounded by a constant times
$$
f(\eps)^{-1} {L\b^{-2}} \epsilon_{0}^{d-1} \delta L\b^{d} \ (L\b\epsilon_{0})^{2-d} = f(\eps)^{-1} \delta \eps_0,
$$
which is the desired result.
\end{proof}
For the random walk $(X_n)_{n \geq 0}$, let $\sigma$ be the stopping time defined by
\begin{equation}
\label{defsigma}
\sigma = \inf \{n \ge 0 : \Ll| X_{n} -X_{0} \Rr|  \geq \epsilon_{0} L\b\}.
\end{equation}
We say that the pair $(x,y) \in (\Z^d)^2$ is \emph{generic} if 
\begin{multline}
\PP_{x}\big[X \text{ hits an important point in } \\
B (X_0, L\b (\epsilon_{0} + 2 \delta)) \setminus B (X_0, L\b(\epsilon_0 - \delta)) \ \vert \ X_{\sigma} = y \big] < \epsilon_{0}^{2} \delta^{1/3}.
\end{multline}
Although this is not explicit in the terminology, we stress that the notion of being generic depends on $\eps_0$ and $\delta$. 
Recalling the definition of the stopping times $(\sigma_j)$ in \eqref{e:defsigmai}, we let 
$$
F = \inf \Ll\{j: X_{\sigma_j} \in \big((M-\eps_0)L\b, \infty \big) \times \IZ^{d-1} \Rr\}
$$
and let $A(M, \epsilon_{0}, \delta)$ be the event that
\begin{enumerate}
\item[(i)] 
the random walk $(X_{\sigma_j})_{j \geq 0 }$ hits $((M-\eps_0) L\b, ML\b) \times \Ll(\frac{\sqrt{M} L\b}{2}, \frac{3\sqrt{M} L\b}{2}\Rr)\times \Ll(-\frac{\sqrt{M} L\b}{2}, \frac{\sqrt{M} L\b}{2}\Rr)^{d-2}$ before leaving $\big(-3L\b, ML\b\big) \times \big(- \sqrt{M} L\b, 2 \sqrt{M} L\b\big)\times \big(- \sqrt{M} L\b,  \sqrt{M} L\b\big)^{d-2}$;
\item[(ii)] $F \le (1+ \epsilon )M \epsilon_{0}^{-2} \sqrt{d}$;
\item[(iii)]
for every $j \le F$, the pair $(X_{\sigma_{j}}, \  X_{\sigma_{j+1}})$ is generic.
\end{enumerate}
As a consequence of Lemmas~\ref{six} and \ref{l:annulus}, we obtain the following result.
\begin{prop} 
\label{cor55}
Let $M, \epsilon_0$ be given with $\epsilon _0 $ small. There exists $\ov{\delta} \in(0,1)$ such that for every
$ \delta < \ov{\delta}$, there exists $\ov{\delta}_1 \in (0, \ov{\delta})$ such that  for every $\delta_1  <  \ov{\delta}_1$, 
if the environment in $(-4   L\b, (M+1)L\b) \times {(-3 \sqrt{M} L\b, 3\sqrt{M} L\b)}^{d-1}$ is $(L\b, \delta_1)$-good, then for every $x \in (-  L\b/2, L\b/2) \times {(- \sqrt{M} L\b/2, \sqrt{M}L\b/2)}^{d-1}$, we have
$$
\PP_x\Ll[A (M, \epsilon_{0}, \delta)\Rr] \geq  \frac{c_1}{3} e ^{-M{\sqrt d}/2}.
$$
\end{prop}

\begin{proof}
Let $K$ ($=K(\eps)$) be given by Lemma~\ref{l:annulus}, and let $0 < \delta < \epsilon _0 /2$. We write $\mcl{E}(\eps_0,\delta)$ for the event that the random walk hits an important point in $B(X_0, L_\lambda (\epsilon_0 + 2\delta ))  \backslash B(X_0, L_\lambda (\epsilon_0 - \delta )) $. By Lemma~\ref{l:annulus}, if the environment in $(-4   L\b, (M+1)L\b) \times {(-3 \sqrt{M} L\b, 3\sqrt{M} L\b)}^{d-1}$ is $(L\b, \delta_1)$-good and $\delta_1$ is small enough, then for any $x $ in  $\big(-3L\b, ML\b\big) \times \big(- \sqrt{M} L\b, 2 \sqrt{M} L\b\big)\times \big(- \sqrt{M} L\b,  \sqrt{M} L\b\big)^{d-2}$, 
$$
\PP_x\Ll[ \mcl{E}(\eps_0,\delta) \Rr] \le K \delta \eps_0.
$$
As a consequence, the probability that 
$$
\PP_x[\mcl{E}(\epsilon_0,\delta)  \ \vert \ X_\sigma ] \geq  \epsilon_0 ^2 \delta^{1/3} 
$$
is bounded by 
$$ 
\frac{ K \delta \epsilon_0}{ \epsilon_0 ^2 \delta^{1/3} } \ = \ \frac{ K  \delta^{2/3}}{ \epsilon_0  } \ < \ \epsilon_0 ^2 \delta^{1/3} ,
$$
provided that $ \delta < \ov{\delta} \leq {\epsilon_0^9}/{K^3}$.
Thus, we have for $x $ as above,
$$
\PP_x[(x, X_\sigma ) \mbox{ is not generic} ] \ < \ \epsilon_0^2 \delta^{1/3}.
$$

Let $\mcl{E}^i(\epsilon_0, \delta) $ be the event that $(X_{\sigma_i}, X_{\sigma_{i+1}})$ is not generic.  By the strong Markov property, 
the probability that there exists $ i \leq (1+\epsilon) M \epsilon_0^{-2} \sqrt{d} $ such that 
$$
\left\{
\begin{array}{l}
X_{\sigma_i}  \in \big(-3L\b, ML\b\big) \times \big(- \sqrt{M} L\b, 2 \sqrt{M} L\b\big)\times \big(- \sqrt{M} L\b,  \sqrt{M} L\b\big)^{d-2} \quad \text{ and } \\ 
 \mcl{E}^i(\epsilon_0, \delta) \text{ occurs}
\end{array}
\right.
$$
is bounded by 
$$ (1+\epsilon) M\sqrt{d}\delta^{1/3} .
$$
The result then follows provided $\ov{\delta}$ is less than $\left( \frac{c_1}{6M \sqrt{d} (1+ \epsilon)} e ^{-M{\sqrt d}/2} \right)^3$, using Lemma~\ref{six} with $h = d^{- 1/2}$. 
\end{proof}
\begin{rem}
We wish to emphasize that the event $A (M, \epsilon_{0}, \delta)$ does not (explicitly) depend on the actual hitting time of the hyperplane $\{M L\b\} \times \IZ^{d-1}$ other than through the rough clock provided by the $\sigma_i$s. 
\end{rem}
\begin{rem}
The part (iii) in the definition of $A (M, \epsilon_{0}, \delta)$
enables us to conclude that with high probability 
on the event $A (M, \epsilon_{0}, \delta)$  the important points 
within $\delta L\b $ of the points $X_{\sigma_i}$ may be ignored.
\end{rem}

The next major point is to examine the killing 
probabilities as the random walk passes from $X_{\sigma_i }$ to $X_{\sigma_{i+1} }$.
We have (see \cite{LL}) that for $T_x \ = \ \inf \{n: X_n = x \}$,
$${|x|}^{d-2} \PP_0 (T_{x} < \infty) \to c$$
for some $c \in (0, \infty )$.
From this it follows that 
\begin{lemma} \label{lll}
Let $\tau_r \ = \ \inf \{ n: |X_n| > r \}$. There exists a continuous $\phi : [0,1] \to \R$, strictly positive on $(0,1)$, such that
for every $\eta > 0$, uniformly on $\frac{\vert x \vert}{r}  \in \ (\eta, 1-\eta)$, we have
$$
|x|^{d-2}\PP_0[T_{x} < \tau_{r}] - \phi \left( \frac{\vert x \vert}{r} \right) \to 0
$$
as $r$ tends to infinity.
\end{lemma}
Indeed, the lemma can be obtained (with an explicit expression for $\phi$) by noting that the law of the random walk when hitting the sphere is asymptotically uniformly distributed, and decomposing the event of touching $x$ according to whether it occurs before or after hitting the sphere. We refer to \cite[Lemma~A.2]{bc} for details. The only important point for us is that since $\EE_0(\tau_{r})/r^{2} \xrightarrow[r \to \infty]{} 1,$ we have 
\begin{equation}
\label{integr1}
\int_{D(0,1)} \frac{\phi(|v|)}{q_d \, |v|^{d-2}} \ {\d v} = 1,
\end{equation}
where $D(0,1)$ is the unit ball in $\IR^d$. The following lemma follows from this observation.

\begin{lemma}
If $\delta_{1}$ is fixed sufficiently small (in terms of $\delta$, $\eps$ and $\eps_0$) and if $B(0,\epsilon_{0} L\b )$ is an $(L\b, \delta_1)$-good environment, then for every $j$ such that $ I_{j}$ is relevant,
\begin{multline*}
\sum_{ \delta L\b \leq \vert x \vert \leq ( \epsilon_{0} - \delta) L\b}
\PP_{0}[X \text{ hits } x \text{ before time } \sigma] \ \1_{V(x) \in I_{j}} \\
\leq (1+3\epsilon)  \epsilon_0^2 L\b^{2}q_d \ \P[V(0) \in I_{j}]  , 
\end{multline*}
while
\begin{equation*}
\sum_{j:I_{j}\text{ is irrel.}} \  \sum_{ \delta L\b \leq \vert x \vert \leq (\epsilon_{0} - \delta_{}) L\b} \PP_{0}[X \text{ hits } x \text{ before time } \sigma] \ \1_{V(x) \in I_{j}} \le 5 \epsilon_0^2 \epsilon ^6.
\end{equation*}
\end{lemma}

\begin{proof}
We begin with the proof concerning relevant intervals. Note first that
\begin{multline}
\label{eq1}
\sum _ {\delta L\b \leq \vert x \vert \leq ( \epsilon_{0} - \delta) L\b}  \PP_0[ T_x < \sigma ] \ \1_{V(x) \in I_{j}} \\
\leq
\sum_{ \substack{\un{n} \in \Z^d : \\ C_{\underline{n},\delta_1} \cap \Ll(B(0, (\epsilon_0 -\delta) L_\lambda) \backslash B(0, \delta L_\lambda) \Rr)\ne \emptyset }} \  \sum _ {x \in C_{\underline{n},\delta_1} : V(x) \in I_j}  \PP_0[ \tau_x < \sigma ] .
\end{multline}
For $ \delta_1 < \delta /3d $, if $C_{\underline{n},\delta_1}  $ intersects $B(0, (\epsilon_0 -\delta) L_\lambda) \backslash B(0, \delta L_\lambda)$,
then necessarily $ C_{\underline{n},\delta_1}  \subset B(0, (\epsilon_0 -\delta/2) L_\lambda) \backslash B(0, \delta L_\lambda / 2)$.  By Lemma~\ref{lll}, 
$$
\Ll| |x|^{d-2} \PP_0[T_x < \sigma] - \phi\Ll(\frac{|x|}{\eps_0 L\b}\Rr) \Rr|
$$ 
tends to $0$ as $\lambda$ tends to $0$, uniformly over all $x$ in $B(0, (\epsilon_0 -\delta/2) L_\lambda) \backslash B(0, \delta L_\lambda / 2)$. These observations imply that the right-hand side of \eqref{eq1} is asymptotically equivalent to
$$
\sum_{ \substack{\un{n} \in \Z^d : \\ C_{\underline{n},\delta_1} \cap \Ll(B(0, (\epsilon_0 -\delta) L_\lambda) \backslash B(0, \delta L_\lambda) \Rr)\ne \emptyset }} \  \sum _ {x \in C_{\underline{n},\delta_1} : V(x) \in I_j} \phi\Ll(\frac{|x|}{\eps_0 L\b}\Rr)|x|^{2-d}
$$
as $\lambda$ tends to $0$. For $\delta_1$ sufficiently small (and since $\phi$ is continuous), this is smaller than 
$$
\sum_{ \substack{\un{n} \in \Z^d : \\ C_{\underline{n},\delta_1} \cap \Ll(B(0, (\epsilon_0 -\delta) L_\lambda) \backslash B(0, \delta L_\lambda) \Rr)\ne \emptyset }} \! \! \! \! \! \! \! \! (1+ {\epsilon} ) \frac{|\{x \in C_{\underline{n},\delta_1} : V(x) \in I_j \}| }{ (\epsilon_0L_\lambda )^{d-2}} \ \frac{1}{|D_{\underline{n},\delta_1}|  }\int_{D_{\underline{n},\delta_1} } \frac{\phi(|v|)}{|v|^{d-2}} \ \d v,
$$
where $D_{\underline{n},\delta_1} $ is the (continuous) rectangle corresponding to $\frac{C_{\underline{n},\delta_1}}{\epsilon_0 L_\lambda }$, and $|D_{\underline{n},\delta_1}|$ is its Lebesgue measure, that is, $(\delta_1/\eps_0)^d$.
Using the fact that
our environment is $(L_\lambda , \delta_1)$-good (see \eqref{goodrel}) and \eqref{integr1}, we obtain that
$$ 
\sum _ {\substack{\delta L\b \leq \vert x \vert \leq ( \epsilon_{0} - \delta) L\b \\ V(x) \in I_j}}  \PP_0[ T_x < \sigma ]  \leq 
(1+ 3 \eps) \epsilon_0^2 L\b^{2}q_d \ \P[V(0) \in I_{j}],
$$
for all $\lambda$ sufficiently small, as announced.

The proof for irrelevant intervals is the same, except that one uses \eqref{goodirr} and the fact that there are no more than $\eps^{-3}$ irrelevant intervals to conclude.
\end{proof}

In fact given the Harnack principle for random walks (see \cite{LL}) one has 

\begin{lemma} \label{har}
Let $ \ \delta, \ \epsilon_{0}, \ \delta_{1}$ be as in the previous lemma. If $B(0,\epsilon_{0} L\b )$ is an $(L\b, \delta_1)$-good environment, then for every $\lambda $ sufficiently small (i.e.\ for $L\b$ sufficiently large) and every $y \in \partial B(0, L\b \epsilon_{1})$,
\begin{multline*}
\sum_{\delta L\b \leq \vert x \vert \leq (\epsilon_0 - \delta) L\b}\PP_{0}[X \text{ hits } x \text{ before time } \sigma \ \vert \ X _\sigma =y] \ \1_{V(x) \in I_{j}} \\
\leq (1+3\epsilon) \epsilon_0^2 L\b^{2}q_d  \ \P[V(0) \in I_{j}]
\end{multline*}
while
\begin{multline*}
\hspace{-0.3cm} \sum_{j: I_{j}\text{ is irrel.}} \ \sum_{\epsilon_0 \delta L\b \leq \vert x \vert \leq (\epsilon_{0}(1 - \delta)) L\b} 
\PP_{0}[X \text{ hits } x \text{ before time } \sigma  \ | \ X _\sigma =y] \ \1_{V(x) \in I_{j}} \\
\le 5 \eps_0^2 \eps^6.
\end{multline*}
\end{lemma}

\begin{proof}[Proof]
The justification of the statements comes down to an analysis of the behaviour of $\PP_0[X \mbox{ hits } x \mbox{ before } \sigma \vert X_\sigma =y]$.  Letting $h^y(u)$ be the harmonic function
$\PP_u[ X_\sigma =y] $ for $ u \in  B(0,\eps_0 L\b)$, we have 
$$
\PP_0[X \mbox{ hits } x\mbox{ before } \sigma \vert X_\sigma =y] = \frac{h^y(x)}{h^y(0)} \PP_0[X \mbox{ hits } x\mbox{ before } \sigma ].
$$
It is enough to show that the quantity $ \frac{h^y(x)}{h^y(0)}  -R(\frac{x}{\epsilon L\b},\frac{y}{|y|}) $ converges to zero as $L\b$ becomes large where $R(.,.) $ is the Riesz kernel  (given the continuity of the Riesz kernel $R(.,.)$) .

The main point is that for $ 0 < \delta' < \delta $ and $ u \in B(0,(\epsilon_0 - \delta)) $ the distributions of 
$$
  \frac{X_{\tau_{(\epsilon_0 - \delta' )L\b}}- y}{ \delta' }
$$
are tight under the laws $\PP_u[ \ \cdot \ | X_\sigma= y]$.  This can be argued as in \cite{law}.   Simple bounds on $h^y (u)$ then show that for $K$ large 
$h^y(u) $ is close to $ \sum_{ |v-y| < K \delta^ \prime} \PP_u[X_{\tau_{(\epsilon_0 - \delta' )L\b}}=v] h^y(v) $.  We now fix $ \delta' \ll 1$ and apply the invariance principle as $L \rightarrow \infty $ in conjunction with the Harnack inequality for $h^y$ to get the desired result. 
\end{proof}


As is well known, for a random walk in dimensions three and higher, given that a point is visited, the number of visits is geometric of parameter $q_d$.
From this, it is almost immediate that if a point is well distanced from the boundary of $B(0, \epsilon _0L\b )$ and is visited by a random walk before time $\sigma$, then the number of visits ought to be approximately geometric with parameter $q_d$.  The following is our formulation of this.
\begin{lemma} \label{geom}
For fixed $ \epsilon_0, \ \delta \ > \ 0$, there exists $ c(L\b,\epsilon_0, \delta ) $ tending to one as $L\b $ tends to infinity such that for all $x  \in  B(0, (\epsilon _ 0 - \delta )L\b ) \setminus B(0, \delta L\b) $,
$$
(1-q_dc(L\b,\epsilon_0, \delta ))^{r-1}  \leq \PP_0\Ll[N(x) \geq r \ | \ N(x) \geq 1\Rr] \le (1-q_d)^{r-1},
$$
where $N(x)$ is the number of visits to site $x$ before time $\sigma$.
\end{lemma}
\begin{proof}
For the right-hand side we simply note that by the strong Markov property, given that the point $x$ is visited by a random walk, the number of visits to site $x$ is geometric with 
parameter $q_d$.  
For the left-hand side inequality, we need a lower bound on the probability that 
the random walk starting at $x$ returns there before time $\sigma$. We can decompose this probability as
$$
1- q_d - \PP_x[X \text{ hits } x \text{ after time } \sigma] \ge 1 - q_d - \sup_{y \in \partial B(0,\eps_0 L\b)} \PP_y[T_x < +\infty],
$$
and this finishes the proof.
\end{proof}
Lemma~\ref{geom} and convexity yield
\begin{prop} \label{p11}
Given $\epsilon_{0}, \delta, \delta_1$ (and $\eps$ fixed small enough), for all $\lambda$ small enough, if $B(0,\eps_0 L\b)$ is $(L\b, \delta_1)$-good, then
\begin{multline*}
\EE_{0} \left[ \exp\Ll({-\sum_{s < \sigma} \lambda V(X_{s})  \1_{V(X_s) \geq \frac{\epsilon}{\lambda }} \1_{X_s \in B(0, (\epsilon _ 0 - \delta )L\b ) \setminus B(0, \delta L\b) }} \Rr) \ \bigg| \  X_\sigma =y \right]\\
\geq 
\exp\Ll(-(1+4\epsilon)L\b^{2}\epsilon_{0}^{2}\int^{\infty}_{\frac{\epsilon}{\lambda}}    f(\lambda z) \ \d \mu(z) \Rr).
\end{multline*}
\end{prop}

\begin{proof}
Since for all $u$ and $v$ positive, $1-e^{-u+v}  \leq 1-e^{-u} + 1-e^{-v}$, we have 
\begin{multline}
\label{e:p11}
1- \EE_{0} \left[ \exp\Ll(-\sum_{s < \sigma} \lambda V(X_{s}) \1_{V(X_s) \geq \frac{\epsilon}{\lambda }} \1_{X_s \in  B(0, (\epsilon _ 0 - \delta )L\b ) \setminus B(0, \delta L\b) }\Rr) \ \bigg| \  X_\sigma =y \right] \\
\leq \sum _{\substack{x \in B(0, (\epsilon _ 0 - \delta )L\b ) \setminus B(0, \delta L\b) \\ V(x) \ge \frac{\eps}{\lambda}}}  \EE_0\Ll[1- e^{-\lambda N(x) V(x)}  \ | \  X_\sigma =y \Rr].
\end{multline}
If we restrict the sum above to those $x$ such that $V(x) \in I_j$, where $I_j = [a_j,b_j)$ is a relevant interval, then we can use Lemmas \ref{har} and \ref{geom} to bound the sum by
\begin{multline*}
(1+3\epsilon) \epsilon_0^2 L\b^{2}q_d  \ \P[V(0) \in I_{j}] \Ll( 1-e^{- \lambda b_j} \sum_{r = 0}^{+\infty} (1-q_d)^r q_d e^{- r \lambda b_j} \Rr) \\
= (1+3\epsilon) \epsilon_0^2 L\b^{2} \ \P[V(0) \in I_{j}] \ q_d \frac{1-e^{- \lambda b_j}}{1-(1-q_d)e^{- \lambda b_j}}.
\end{multline*}
Similarly, in the right-hand side of \eqref{e:p11}, the sum restricted to those $x$ such that $V(x)$ belongs to some irrelevant interval is bounded by $5 \eps_0^2 \eps^6$. In total, the right-hand side of \eqref{e:p11} is thus bounded by
$$
5 \eps_0^2 \eps^6 + (1+3\epsilon) \epsilon_0^2 L\b^{2} \sum_{j : I_j \text{ is rel.}} \P[V(0) \in I_{j}] \ q_d \frac{1-e^{- \lambda b_j}}{1-(1-q_d)e^{- \lambda b_j}}.
$$
Using \eqref{c:assumption1}, one can check that this is smaller than 
$$
\Ll(1+\frac{7\epsilon}{2}\Rr)L\b^{2}\epsilon_{0}^{2}\int^{\infty}_{\frac{\epsilon}{\lambda}}    \frac{ q_d (1- e^{- \lambda z} ) }{1-(1- q_d) e^{- \lambda z}} \ \d \mu(z),
$$
and the result is then obtained provided $\eps_0$ is sufficiently small, since, by the definition of $L\b$ in \eqref{defL},
$$
L\b^{2} \int^{\infty}_{\frac{\epsilon}{\lambda}}    \frac{ q_d (1- e^{- \lambda z} ) }{1-(1- q_d) e^{- \lambda z}} \ \d \mu(z) \le 1.
$$
\end{proof}


\begin{cor}  
\label{cor56}
For $ \epsilon$, $\epsilon_0 $ and $ \delta $ small enough and then $\delta_1$ chosen small enough, if 
\begin{equation}
\label{e:cor56}
(-4   L\b, (M+1)L\b) \times {(-3 \sqrt{M} L\b, 3\sqrt{M} L\b)}^{d-1}
\end{equation}
is $(L\b, \delta_1)$-good, then for every $x \in (-L\b/2,L\b/2)\times(-\sqrt{M}L\b/2,\sqrt{M}L\b/2)^{d-1}$, 
\begin{multline}
\label{l:cor56}
\EE_x \left[ \exp \left( -\sum _{j \in J } \lambda V(X_j) \1_{V(X_j) \geq \frac{\epsilon}{\lambda}}   \right) \ \bigg| \ A(M, \epsilon_0, \delta ) , \ F, \ (X_{\sigma_i})_{i \leq F} \right]  \\
\ge \exp\left( -M (1+6 \epsilon) L\b^{2}  \sqrt{d} \int_{ \epsilon /  \lambda }^{\infty}   f(\lambda z) \ \d \mu(z)  \right),
\end{multline}
where $J $ is the set of times $ 0 \leq j \leq \sigma_F $ such that if
$\sigma_{i-1} \leq j \leq \sigma_i $ then $X_j \in B(X_{\sigma_{i-1}}, (\epsilon_0 -\delta)L\b ) \setminus B(X_{\sigma_{i-1}},  \delta L\b )$.
\end{cor}
\begin{proof}
By Proposition \ref{p11}, we have that 
\begin{multline*}
\EE_0 \left[ \exp \left( -\sum _{j \in J } \lambda V(X_j) \1_{V(X_j) \geq \frac{\epsilon}{\lambda}}   \right) \ \bigg| \ A(M, \epsilon_0, \delta ) , \ F, \ (X_{\sigma_i})_{i \leq F} \right]  \\
\ge \left( \exp\Ll(-(1+4\epsilon)L\b^{2}\epsilon_{0}^{2}\int^{\infty}_{\frac{\epsilon}{\lambda}}   f(\lambda z)  \ \d \mu(z) \Rr) \right) ^F.
\end{multline*}
Since on the event $A(M, \epsilon_0, \delta ) , \ F \leq (1+ \epsilon)M \sqrt{d} \epsilon_0 ^{-2}$, this is bounded below by 
$$
\left( \exp\Ll(-(1+4\epsilon)L\b^{2}\epsilon_{0}^{2}\int^{\infty}_{\frac{\epsilon}{\lambda}}   f(\lambda z) \ \d \mu(z) \Rr) \right) ^{ (1+ \epsilon)M \sqrt{d} \epsilon_0 ^{-2}}.
$$
This gives the result provided $ \epsilon $ was fixed sufficiently small.
\end{proof}

We now show that in the left-hand side of \eqref{l:cor56}, one can replace the sum over $j \in J$ by the sum of the same summands over all $j$, and moreover, one can remove the restriction on $V(X_j) \ge \frac{\eps}{\lambda}$. Recall that  $V(X_j)1_{V(X_j) \geq \frac{\epsilon}{\lambda}} = \tilde V_{\epsilon }(X_j)1_{\tilde V_{\epsilon }(X_j) \geq \frac{\epsilon}{\lambda}} $.

\begin{cor}  \label{cor57}
Under the conditions of Corollary~\ref{cor56}, with probability at least $1-\delta ^{1/8} $,
\begin{multline*}
\EE_x \left[ \exp \left( -\sum _{j \in J^\prime } \lambda \tilde V_{\epsilon }(X_j)     \right) \1_{\sigma_F \leq M \sqrt {d} L\b^2 (1+ 5 \epsilon)} \ \bigg| \ A(M, \epsilon_0, \delta ), \ F, (X_{\sigma_i})_{i \leq F}  \right] \\
\geq  \exp\left( -M (1+7\epsilon)L\b^{2}  \sqrt{d} \ \mfk{I}_{\eps,\lambda} \right),
\end{multline*}
where $J' $ is the collection of $j \le \sigma_F$ such that $X_j \notin B(0, \delta L\b)$.  If in addition the probability of hitting an important site within 
$ \delta  L\b $ of $0$ is less than $ \delta ^ {1/8}$, then this bound extends to all summands $j$ with $j \leq \sigma_F$.
\end{cor}

\begin{proof}
For any event $\mcl{A}$ and any positive random variable $Z$,
\begin{multline}
\label{obs1}
\EE_0[e^{-Z} \1_\mcl{A} {\ \vert \ }   A(M, \epsilon_0, \delta ),F, (X_{\sigma_i})_{i \leq F} ] \\
 \geq \EE_0[e^{-Z} { \ \vert \ }   A(M, \epsilon_0, \delta ),F, (X_{\sigma_i})_{i \leq F} ]-\PP_0[\mcl{A}^c{ \ \vert }  \  A(M, \epsilon_0, \delta ),F, (X_{\sigma_i})_{i \leq F} ].
\end{multline}
We will apply this for the
random variable $Z$ equal to $ \sum _{j \in J^\prime } \lambda V(X_j) \1_{V(X_j) \ge \epsilon / \lambda }$ and the event
$\mcl{A}$ taken to be that
\begin{itemize}
\item no important point within $\delta \epsilon_0 L\b$ of a point $X_{\sigma_i}$  for some $1 \le i \leq F$ is hit
\item and $\sigma_F \leq M \sqrt {d} L\b^2 (1+ 5 \epsilon)$.
\end{itemize}
To begin with, we observe that on the event $\mcl{A}$, since $\sigma_F \leq M \sqrt {d} L\b^2 (1+ 5 \epsilon)$, we have almost surely
\begin{multline}
\exp\Ll(- \sum _{j \in J^\prime } \lambda \tilde V_{\epsilon }(X_j) \1_{\tilde V_\epsilon (X_j) < \epsilon/\lambda} \Rr) \\
\geq \exp\Ll(-M {\sqrt d }L\b^2 (1+ 5 \epsilon ) \ \lambda \E[V(0) \ | \ {V(0) < \epsilon / \lambda} ] \Rr) .
\end{multline}
Hence,
\begin{equation}
\label{e:c310}
\EE_0 \left[ \exp \left( -\sum _{j \in J^\prime } \lambda \td{V}_\eps(X_j)   \right) \1_\mcl{A} \ \bigg| \ A(M, \epsilon_0, \delta ), \ F, (X_{\sigma_i})_{i \leq F}  \right] 
\end{equation}
is larger than
\begin{multline*}
\exp\Ll(-M {\sqrt d }L\b^2 (1+ 5 \epsilon ) \ \lambda \E[V(0) \ | \ {V(0) < \epsilon / \lambda} ] \Rr) \\
\times \EE_0 \left[ \exp \left( -\sum _{j \in J^\prime } \lambda V(X_j)1_{V(X_j) \ge \epsilon / \lambda }    \right) \1_\mcl{A} \ \bigg| \ A(M, \epsilon_0, \delta ), \ F, (X_{\sigma_i})_{i \leq F}  \right].
\end{multline*}
Using \eqref{obs1} and Corollary~\ref{cor56}, we get that the latter conditional expectation is larger than
$$
\exp\Ll(- M (1+6\epsilon)L\b^{2}  \sqrt{d}\int_{\epsilon/ \lambda }^{\infty}  f(\lambda z)  \ \d \mu(z)  \Rr) 
-  \PP_0[\mcl{A}^c{ \ \vert }  \  A(M, \epsilon_0, \delta ),F, (X_{\sigma_i})_{i \leq F} ].
$$
One can check that for $\lambda$ sufficiently small,
$$
(1+5\eps) \lambda \E[V(0) \ | \ {V(0) < \epsilon / \lambda} ] \le (1+6\eps) \int_0^{\eps/\lambda} f(\lambda z) \ \d \mu_\eps(z),
$$
so that the conditional expectation in \eqref{e:c310} is larger than
$$
\exp\Ll(- M (1+6\epsilon)L\b^{2}  \sqrt{d} \ \mfk{I}_{\eps,\lambda} \Rr)-  \PP_0[\mcl{A}^c{ \ \vert }  \  A(M, \epsilon_0, \delta ),F, (X_{\sigma_i})_{i \leq F} ].
$$
From the proof of Proposition \ref{cor55}, we learn that reducing $\epsilon_0$ if necessary, 
$$
\PP_0 [\mcl{A}^c \ | \  A(M, \epsilon_0, \delta )] \ \leq \ \delta^{1/4}
$$ 
for $\delta $ small. As a consequence,
$$
\ \PP_0 [\mcl{A}^c{ \ \vert }  \ A(M, \epsilon_0, \delta ), F, (X_{\sigma_i})_{i \leq F} ] \ \leq \ \delta ^{1/8}
$$
outside an event of probability less than $\delta^{1/8}$. Outside this event, the conditional expectation in \eqref{e:c310} is larger than
$$
\exp\Ll(- M (1+6\epsilon)L\b^{2}  \sqrt{d} \ \mfk{I}_{\eps,\lambda} \Rr)- \delta^{1/8}.
$$
In view of the definition of $L\b$ in \eqref{defL}, it is clear that it suffices to choose $\delta$ sufficiently small to ensure that this is larger than
$$
\exp\Ll(- M (1+7\epsilon)L\b^{2}  \sqrt{d} \ \mfk{I}_{\eps,\lambda} \Rr),
$$
and the proof is complete.
\end{proof}

{\it Definition:} Given $\epsilon $ and $ \lambda $, we say that a point $x \in {\IZ} ^d$ is $(L\b, \epsilon_0, \delta )$\emph{-healthy} if the probability for the random walk started at $x$ to hit an important point within $\epsilon_0 \delta L\b$ of $x$ is less than $ \delta ^{1/4} $.

This and Proposition~\ref{cor55} give
\begin{prop}  \label{stage}
Under the conditions of Corollary~\ref{cor56}, 
\begin{multline*}
\EE_x \left[ \exp \left( -\sum _{j \leq \sigma } \lambda \tilde V_{\epsilon }(X_j)    \right) \1_{A(M,\epsilon,\delta)}\right]  \\
\geq \exp\left( -M (1+8\epsilon)L\b^{2}  \sqrt{d}\ \mfk{I}_{\eps,\lambda}\right)
\frac{c_1}{3} e^{-M \sqrt d/2},
\end{multline*}
provided that $x$ is $(L\b, \epsilon_0, \delta )$-healthy.
\end{prop}

%
%
%
%
%
%
%
\section{The  Block Argument}
We have established Proposition \ref{stage} above which basically states that (with appropriate probability) a random walk starting in $( -L\b\sqrt M/2 , L\b \sqrt M /2) ^ d $
will arrive at $\{ML\b\} \times ( L\b\sqrt M/2 , 3L\b \sqrt M /2) \times ( -L\b\sqrt M/2 , L\b \sqrt M /2) ^{ d-2} $ without leaving the designated bounded area provided the 
environment is $(L\b, \delta_1)$-good.  This motivates an oriented percolation approach to show that with high probability as $L\b$ tends to infinity, there is the possibility of the random walk travelling from $(0,0, \cdots 0) $ to $(n,0, \cdots 0) $ without entering bad environments and by essentially having the first component increase by $ML\b$ in time intervals of 
length $ML\b^2 d ^ {1/2} $.   

For $i,j \in \Z$, let $B_{i,j}$ denote the set 
\begin{multline*}
\big(iML\b - 4L\b,(iM+M+1)L\b\big) \times \big((j-3)\sqrt{M}L\b,(j+3)\sqrt{M}L\b\big) \\
\times \big(-3\sqrt{M}L\b,3\sqrt{M}L\b\big)^{d-2}.
\end{multline*}
The set $B_{i,j}$ is nothing but a translation of $B_{0,0}$, more precisely,
$$
B_{i,j} = (iML\b,j\sqrt{M} L\b, 0\cdots,0) + B_{0,0},
$$
and $B_{0,0}$  is the set appearing in \eqref{e:cor56}.

Let $G = (V,E)$ be the oriented graph with vertex set

$$V= \{(i, j){\in \IZ^{2}}, i \ge 0, i+j \equiv 0 \ \text{mod}\ 2\},$$
and edge set 
$$E= \big\{[(i, j), (i+1, j+1)], \  [(i,j), (i+1, j-1)], \text{ for } (i,j) \in V\big\}.$$
We consider a site percolation process on $G$ by declaring the vertex $(i,j)$ to be \emph{open} if $B_{i,j}$ is $(L\b,\delta_1)$-good. 
\begin{lemma}
\label{l:perco}
Let $A_N$ be the event that there exist $j_0 = 0, j_1, \ldots, j_N$ such that the sequence of sites $(0,j_0), \ldots, (N,j_N)$ is a directed open path in $G$. We have
$$
\inf_N \P[A_N] \xrightarrow[\lambda \to 0]{} 1.
$$
\end{lemma}
\begin{proof}
It follows from Lemma~\ref{goodenv} that the probability for a given site to be open tends to $1$ as $\lambda$ tends to $0$. Moreover, the percolation process has a finite range of dependence. The lemma is then a direct consequence of \cite[Theorem~B26]{lig99} (or of \cite{LSchS}).
\end{proof}
For the random walk $(X_r)_{r \geq 0}$, we define the stopping times $\sigma^i_j$ recursively in the following way. We let $\sigma_0^0 = 0$, and for all $i \geq 0 $ and $j > 0$,
$$
\sigma^i_j \ = \ \inf \{ r > \sigma^i_{j-1}: |X_r- X_{\sigma^i_{j-1}}| \geq \epsilon_0L\b \},
$$
where for $i \ge 0$, the stopping time $\sigma^{i+1}_0$ equals $\sigma^{i}_{F(i)}$ with
\begin{equation}
\label{defki}
F(i)  := \inf \{j: X_{\sigma^{i}_j} \in [(i+1)ML\b- \epsilon_0 L\b, \infty ) \times \IZ^{d-1} \}.
\end{equation}

On the event that $A_N$ is realized, we pick (in some arbitrary way) $j_0=0, j_1,\ldots, j_N$ such that $(0,j_0), \ldots, (N,j_N)$ is a directed open path in $G$. We define the (random walk) event $B_n$ as the conjunction over all $i \le N-1$ of:
\begin{itemize}
\item[(i)] the random walk $(X_{\sigma^i_j})_{j \ge 0}$  hits 
\begin{multline*}
\big((iM + M-\eps_0) L\b, ML\b\big) \times \Ll(\Ll(j_{i+1} - \frac{1}{2}\Rr)\sqrt{M} L\b, \Ll(j_{i+1} + \frac{1}{2}\Rr)\sqrt{M} L\b\Rr) \\
\times \Ll(-\frac{\sqrt{M} L\b}{2}, \frac{\sqrt{M} L\b}{2}\Rr)^{d-2}
\end{multline*}
before leaving 
\begin{multline*}
\big(iML\b-3L\b, (i+1)ML\b\big) \times \big((j_{i+1}-2)\sqrt M L\b, (j_{i+1}+2)\sqrt M L\b \big) \\
\times \big(-2 \sqrt M L\b, 2 \sqrt M L\b\big)^{d-2};
\end{multline*}  
\item[(ii)] $F(i)$ defined in \eqref{defki} satisfies 
$$
F(i) \le (1+\eps) M \eps_0^{-2} \sqrt{d};
$$
\item[(iii)] for every $j \le F(i)$, the pair $(X_{\sigma^i_j},X_{\sigma^i_{j+1}})$ is generic.
\end{itemize}
This event is a conjunction of events similar to the event $A(M,\eps_0,\delta)$.
Since with probability tending to $1$ as $N$ tends to infinity, the origin is healthy, we can apply Proposition~\ref{stage} iteratively and get that on this event and on $A_N$,
\begin{align*}
& \EE_0\Ll[\exp\Ll( \sum_{i \le \tau_{NML\b}} \lambda \td{V}_\eps(X_i) \Rr)  \Rr] \\
& \qquad \ge \EE_0\Ll[\exp\Ll( \sum_{i \le \tau_{NML\b}} \lambda \td{V}_\eps(X_i) \Rr) \1_{B_N} \Rr] \\
& \qquad \ge \Ll(\frac{c_1}{3} \exp\left( -M (1+8\epsilon)L\b^{2}  \sqrt{d}\ \mfk{I}_{\eps,\lambda}\right)  e^{-M \sqrt d/2} \Rr)^{N}.
\end{align*}
Up to multiplicative corrections that can be taken as close to $1$ as desired, the cost of travel to the hyperplane in the direction $(1,0,\cdots,0)$ at distance $n$ to the origin is thus no more than $n$ times
$$
\frac{1}{L\b} \Ll(L\b^2 \sqrt{d} \ \mfk{I}_{\eps,\lambda} + \frac{\sqrt{d}}{2} \Rr)= L\b \sqrt{d}\ \mfk{I}_{\eps,\lambda}  + \frac{\sqrt{d}}{2L\b},
$$
as $n$ tends to infinity. In view of the definition of $L\b$ in \eqref{defL}, this is equal to $\sqrt{2 d \ \mfk{I}_{\eps,\lambda}}$, and the proof is complete.

%
%
%
%
%
%
%
\section{The remaining case}

We now treat the case where the contribution to $\mfk{I}_\lambda$ from the mass on $[\epsilon/ \lambda , \infty )$ is small, i.e.\ when \eqref{assumption1} does not hold, that is,
$$
2 \,  \mfk{I}_{\eps,\lambda} \geq \frac{1}{\epsilon ^ 4} \P(V \geq \epsilon/ \lambda ).
$$

Long but elementary calculations yield that under this condition we have 

$$
 \mfk{I}_\lambda (1-K \epsilon) \ \leq \ \int _0 ^{ \epsilon / \lambda }    f(\lambda z)  \ \d \mu(x) \leq \mfk{I}_\lambda (1+ K \epsilon)
$$

for universal $K$ and equally for universal $K' $

$$
\P(V(0) \geq { \epsilon / \lambda } ) \ \leq \ K' \epsilon ^ 4\int _0 ^{ \epsilon / \lambda }    f(\lambda z)   \ \d \mu(x) .
$$

We now define (a bit simpler than  before) a subset $A \subset \IZ^d $ to be $(L\b, \delta_1)$-good (for an environment) if for every cube of the form
$[i_1 \delta_1 L\b, (i_1+1) \delta_1L\b) \times [i_2 \delta_1 L\b, (i_2+1) \delta_1L\b) \times [i_d \delta_1 L\b, (i_d+1) \delta_1L\b)$ we have that the number of important points is less
than $2 L\b^{d-2} \delta ^dK'  \epsilon ^4 $.  We clearly have 
\begin{lemma} \label{grainy}
There exists $ \epsilon_F $ such that for all $ \epsilon < \epsilon_F$, all $ \delta_1 > 0$ and all $\gamma > 0$, if $\lambda$ is sufficiently small, then the cube $[0,L\b]^d$ is $(L\b, \delta_1)$-good with
probability greater than $1- \gamma $.
\end{lemma}
We can now pursue the arguments of Sections 2 and 3 (with the stopping times $\sigma_i $ as before).  In particular given $0 <  \epsilon _ 0 \ll \epsilon ,$ we say that a point $ x \in \Z^d $ is
$(\epsilon_0, \delta)$ {\it-good } (for the environment given) if
$$
\PP_{x} ( \exists r \leq \sigma : V( X_r) \geq \epsilon / \lambda ) \leq \epsilon_1 \delta,
$$
with $\sigma$ defined as in \eqref{defsigma}.

\begin{prop} \label{properly}
Given $M, \epsilon_0$ and $ \delta $, if $\delta_1$ is fixed sufficiently small and if the environment in $[-3   L\b, ML\b] \times {(-3 \sqrt{M} L\b, 3\sqrt{M} L\b)}^{d-1}$ is $(L\b, \delta_1)$-good, then
$$
\PP_{x} \Ll[A (M, \epsilon_{0}, \delta)\Rr] \geq  \frac{c_1}{3} e ^{-M\sqrt{d}/2}, 
$$
uniformly over initial points $ x \in {(- \sqrt{M} L\b/2, \sqrt{M}L\b/2)}^{d}$ which are $(\epsilon_1, \delta)$-good, where the
event $A (M, \epsilon_{0}, \delta)$ is the intersection of

\begin{enumerate}
\item[(i)] 
the random walk hits $\{M L\b\} \times (\sqrt{M} L\b/2, 3\sqrt{M} L\b/2)\times {(\sqrt{M} L\b/2, \sqrt{M} L\b/2)}^{d-2}$ without leaving $[-3L\b, ML\b] \times (- \sqrt{M} L\b, 2 \sqrt{M} L\b)\times {(- \sqrt{M} L\b,  \sqrt{M} L\b)}^{d-2}$
\item[(ii)]
the random walk hits $\{M L\b\} \times (\sqrt{M} L\b, 3\sqrt{M} L\b)\times {(\sqrt{M} L\b/2, \sqrt{M} L\b/2)}^{d-2}$ at time $\chi_M \leq   M d^{3/2}$
without having hit an important point.
\item[(iii)]
The hitting point $X_{\chi_M} $ is $(\epsilon_0, \delta )$-good.
\end{enumerate}
\end{prop}

From this the argument proceeds as in the preceding sections.

%
%

\end{document}